\documentclass[leqno,12pt]{article} 
\usepackage{amsmath, amssymb}
\usepackage{amsthm} 
\usepackage{amssymb}
\usepackage{mathrsfs}
\usepackage{amssymb, url, color, graphicx, amscd, mathrsfs}
\usepackage[colorlinks=true, bookmarks=true, pdfstartview=FitH, pagebackref=true, linktocpage=true, linkcolor = magenta, citecolor = blue]{hyperref}
\usepackage{graphicx}
\usepackage{times}
\theoremstyle{plain} 
\newtheorem{theorem}{\indent\sc Theorem}[section]
\newtheorem{lemma}[theorem]{\indent\sc Lemma}
\newtheorem{corollary}[theorem]{\indent\sc Corollary}

\theoremstyle{definition} 

\newtheorem{remark}[theorem]{\indent\sc Remark}
\newtheorem{example}[theorem]{\indent\sc Example}

\numberwithin{equation}{section}
\DeclareMathOperator{\Ric}{Ric}

\makeatletter
\def\address#1#2{\begingroup
\noindent\parbox[t]{7.8cm}{%
\small{\scshape\ignorespaces#1}\par\vskip1ex
\noindent\small{\itshape E-mail address}%
\/: #2\par\vskip4ex}\hfill%
\endgroup}%
\makeatother
\title{{Sharp gradient estimates for a heat equation in Riemannian manifolds}} 
\author{
\textsc{Ha Tuan Dung, Nguyen Thac Dung} 
}
\date{} 
\begin{document}

\maketitle


	\begin{abstract}
		In this paper, we prove sharp gradient estimates for a positive solution to the heat
		 equation $u_t=\Delta u+au\log u$ in complete noncompact Riemannian manifolds. 
		 As its application, we show that if $u$ is a
		 positive solution of the equation $u_t=\Delta u$ and $\log u$ is of sublinear growth
		 in both spatial and time directions then $u$ must be constant.  This gradient estimate is sharp since it is well-known that 
		 $u(x,t)=e^{x+t}$ satisfying $u_t=\Delta u$. We also emphasize that our results are better than those given by Jiang (\cite{XJ16}),  Souplet-Zhang (\cite{SZ06}), Wu (\cite{Wu15, Wu17}), and others.
		
	\end{abstract}
\section{Introduction}
	In 1993, Hamilton \cite{Ham93} proved a gradient estimate for a positive solution $u$ to the heat equation 
\begin{equation}\label{elinear}
u_t=\Delta u
\end{equation}
in complete compact Riemannian manifolds with the Ricci curvature bounded from below by $-K, (K\leq0)$. If $u\leq A$ then it was verified by Hamilton that
	$$\frac{|\nabla u|}{u}\leq\left(\sqrt[]{2K}+\frac{1}{\sqrt[]{t}}\right)\sqrt[]{\log\left(\frac{A}{u}\right)}.$$
	Motivated by Hamilton's result, in 2006, Souplet and Zhang \cite{SZ06} introduced a new gradient estimate for the heat equation \eqref{elinear} in complete non-compact Riemannian manifolds. Assume that the Ricci curvature bounded from below $-K (K\geq0)$ and $u$ is a positive solution to the heat equation in $Q_{R, T}\equiv B(x_0, R)\times[t_0-T, t_0]\subset M\times\mathbb{R}$, they showed that if $u\leq A$ in $Q_{R,T}$ then there exists a dimensional constant $c(n)$ such that the following local gradient estimate 
	$$\frac{|\nabla u|}{u}\leq c(n)\left(\frac{1}{R}+\frac{1}{\sqrt[]{T}}+\sqrt[]{k}\right)\left(1+\log\frac{A}{u}\right)$$
	holds true. As an its consequence, Souplet and Zhang obtained a Liouville theorem for positive ancient solutions $u$ of the heat equation provided that $u=e^{o(d(x)+\sqrt[]{t})}$. Since $u=e^{x+t}$ satisfies $u_t=\Delta u$, their growth condition in the spatial direction is sharp. However, it is clear that their gradient estimate is not sharp in the time direction by the same example. Recently, gradient estimate of Hamilton type or Souplet-Zhang type are extendly studied in several works. For example, as pointed out by google scholar, there are more than  one hundred papers cited Souplet-Zhang's gradient estimate. For further details, we also refer the reader to \cite{DKN18, Ham93, SZ06, Wu15, Wu17} and the references there in. On the other hand, we also remark that Souplet-Zhang gave an example in \cite{SZ06} showing that Hamilton's estimate for the compact case cannot be extended directly, in a localized form, to noncompact manifolds. Moreover, as Souplet-Zhang's observation, it has been known that the right hand side term in Souplet-Zhang's gradient estimate is slightly bigger
than that in Hamilton's theorem (in the power of the log term).
	
In this paper, inspired by Hamilton's idea and Souplet-Zhang's work, we want to complete the picture about sharp gradient estimates in time direction for positive solution of the linear heat equation $u_t=\Delta u$ on complete non-compact Riemannian manifolds and local Hamilton's gradient for complete noncompact manifolds. For this purpose, our first aim is to point out a gradient estimate for a positive solution to the following equation
	\begin{equation}\label{eq:1.1}
	u_t=\Delta u+au\log u,
	\end{equation}
where $a$ is a real constant, on complete non-compact Riemannian manifolds. As showed later, due to Souplet-Zhang's example, to extend Hamilton's gradient estimate to the non-compact case, we need to insert necessary correction term in the right hand side. However, as good as Souplet-Zhang's gradient estimate, our estimate holds for noncompact manifolds and it also has a localized version as the Cheng-Yau estimate and the Hamilton's gradient estimate. The second aim of this paper is to derive a sharp gradient estimate in both spatial and time directions. Namely, we will prove that the Liouville property holds true if a positive solution $u$ of the equation $u_t=\Delta u$ is of growth $e^{o(d(x)+|t|)}$. This Liouville result confirms that our gradient estimate is sharp, exactly. Now, this is suitable time for us to introduce our main theorem.
\begin{theorem}\label{tr:1}
Let $\left( M^n,g \right)$ be an $n$-dimensional  complete Riemannian manifold with
$$\Ric\ge -\left( n-1 \right)K$$
for some constant $K\ge 0$ in $B\left(x_0, R\right)$ some fixed point $x_0$ in $M,$ and some fixed
radius $R\ge 2.$ Assume that $0<u\left(x,t\right)< A$ for some constant $A,$ is a smooth solution to the general heat equation \eqref{eq:1.1} in the cylinder ${{Q}_{R,T}}:\equiv B\left( {{x}_{0}},R \right)\times \left[ {{t}_{0}}-T,{{t}_{0}} \right]\subset M\times \left( -\infty ,\infty  \right),$ where $t_0\in R$ and $T>0,$ then there exists a constant $c\left(n\right)$ such that
\begin{align}\label{eq: 1.2}
\frac{\left| \nabla u \right|}{u}\le c\left( n \right)\left( \frac{\sqrt{\log A-\log \left( \underset{{{Q}_{R,T}}}{\mathop{\inf }}\,u \right)}}{R}+\frac{1}{\sqrt{T}}+\sqrt[4]{\frac{K}{{{R}^{2}}}}+\sqrt{H} \right)\sqrt{\log \frac{A}{u}}
\end{align}
in ${{Q}_{\frac{R}{2},\frac{T}{2}}},$ where 
$$H=\max \left\{ 2\left( n-1 \right)K+a,0 \right\}+\max \left\{ 2a\log A,0 \right\}.$$
\end{theorem}
Note that if $T\geq 2t_0$, then \eqref{eq: 1.2} holds true at $\big(x_0, -\frac{T}{2}\big)$. From this observation, we obtain local a Hamilton's gradient estimate by using \eqref{eq: 1.2}. As applications of Theorem \ref{tr:1}, we have the following Liouville theorems.
\begin{corollary}\label{liouville1}
 Let $M$ be a complete, noncompact manifold with nonnegative
Ricci curvature. If $u$ is a positive ancient solution to the linear heat equation \eqref{elinear} (that is, a solution defined in all space and negative time) such that $u(x,t)=e^{o(d(x)+|t|)}$ near infinity. Then $u$ is a constant.
\end{corollary}
Note that $u(x,t)=e^{x+t}$ is a positive solution to $u_t=\Delta u$. Corollary \eqref{liouville1} implies that our gradient estimate is sharp in both spatial and time directions. Therefore, this results is better than those in \cite{DKN18, SZ06, Wu15, Wu17}. It is also worth to mention that such Liouville theorem still holds true for eternal solution of the heat equation \eqref{elinear} (that is, a solution defined in all space and positive time).

Now, we emphasize that our gradient estimate is not only better than Souplet-Zhang's gradient estimate for positive soultion $u$ of \eqref{elinear} but also better than recent gradient estimates for positive solutions of the general heat equation $u_t=\Delta u+au\log u$, where $a\neq 0$. In fact, we obtain the following Liouville type result.
\begin{corollary}\label{liouville2}
	Let $\left( M^n,g \right)$ be an $n$-dimensional  complete Riemannian manifold with
	$\Ric\ge -\left( n-1 \right)K$
	for some constant $K\ge 0.$  Suppose $u$ is a positive and bounded solution
	to the equation \eqref{eq:1.1} with $a\le 0$ and $u$ is independent of time. Assume $1\le u \le A.$ If $a\le -2\left(n-1\right) K$ then $u\equiv 1.$
\end{corollary}
\begin{remark}
 Note that if $a\le 0$ and $1\le u \le A,$ using the inequality $\log x \le x$ for all $x\ge 1,$ then \eqref{eq: 1.2}  can be rewritten as following 
$$\frac{\left| \nabla u \right|}{{u}^{1/2}} \le c\left( n \right)\sqrt{A}\left( \frac{\sqrt{\log A}}{R}+\frac{1}{\sqrt{T}}+\sqrt[4]{\frac{K}{{{R}^{2}}}}+\sqrt{\max \left\{ 2\left( n-1 \right)K+a,0 \right\}} \right).$$
Therefore, our estimate is better than those in \cite{XJ16, Wu17} for the case $a\le 0.$
Moreover, the Liouville type results we obainned in the case $a\le 0$  for the equation \eqref{eq:1.1} are also better than those in \cite{DKN18, HM15, Yili, XJ16, Wu17}. 
\end{remark}
Finally, by using the same argument in the proof of Theorem \ref{tr:1} and a Laplacian comparison theorem given by Wei and Wylie \cite{WW09} (see also \cite{Bri}), it is worth to notice that our gradient estimates are also valid for a positive solution to 
$$u_t=\Delta_fu+au\log u$$
on smooth metric measure spaces $(M, g, e^{-f}dv)$, where $f$ is the weighted function (see Remark \ref{remark}). Hence, we can improve a Liouville property in Theorem 3.2 by Wu (\cite{Wu15}). Moreover, due to Example 1.2 in \cite{Wu15}, our gradient estimate is sharp. The paper has two sections. Beside this section, we use section 2 to prove Theorem \ref{tr:1} and its corollaries.
\section{Sharp gradient estimates}
	\setcounter{equation}{0}
	In this section, we will give a proof of Theorem \ref{tr:1}. 
	By now a standard routine, we need two basic lemmas. To begin with, let us introduce some notations. Consider the nonlinear heat equation
	\begin{equation}\label{eq:2}
	{{u}_{t}}=\Delta u+au\log u,
	\end{equation}
	where $a$ is a real constant, on an an $n$-dimensional  complete Riemannian manifold $\left( M^n,g \right).$ Suppose that $u\left(x,t\right)$ is a solution of \eqref{eq:2} and $0<u\le A$ for some constant $A$  in the cylinder 
	$${{Q}_{R,T}}:\equiv B\left( {{x}_{0}},R \right)\times \left[ {{t}_{0}}-T,{{t}_{0}} \right]\subset M\times \left( -\infty ,\infty  \right)$$
	where $t_0\in\mathbb{R}$ and $T>0.$ We introduce a new smooth function
	$$h=\sqrt{\log \frac{A}{u}}\ge 0$$
	in ${{Q}_{R,T}}.$ 
	 Observe that
	$$u=A{{e}^{-{{h}^{2}}}}\quad\text{and}\quad\log u=\log A-{{h}^{2}}.$$
	This implies,
	$${{u}_{t}}=-2A{{h}_{t}}h{{e}^{-{{h}^{2}}}}, \quad \nabla u=-2Ah\nabla h{{e}^{-{{h}^{2}}}},$$
	and
	\begin{align}
		\Delta u=\nabla \left( \nabla u \right)&=\nabla \left( -2Ah\nabla h{{e}^{-{{h}^{2}}}} \right)\nonumber\\
		&=-2A\left[ \nabla \left( h\nabla h \right){{e}^{-{{h}^{2}}}}+h\nabla h\nabla \left( {{e}^{-{{h}^{2}}}} \right) \right]\nonumber\\
		&=-2A\left[ {{\left| \nabla h \right|}^{2}}{{e}^{-{{h}^{2}}}}+h\Delta h{{e}^{-{{h}^{2}}}}-2{{h}^{2}}{{\left| \nabla h \right|}^{2}}{{e}^{-{{h}^{2}}}} \right]\nonumber\\
		&=-2Ah{{e}^{-{{h}^{2}}}}\left[ \Delta h+{{\left| \nabla h \right|}^{2}}\left( \frac{1}{h}-2h \right) \right].\nonumber
	\end{align}
	From the equation \eqref{eq:2}, we obtain
	\begin{align}
		-2A{{h}_{t}}h{{e}^{-{{h}^{2}}}}&=-2Ah{{e}^{-{{h}^{2}}}}\left[ \Delta h+{{\left| \nabla h \right|}^{2}}\left( \frac{1}{h}-2h \right) \right]+2Ah{{e}^{-{{h}^{2}}}}\left\langle \nabla f,\nabla h \right\rangle\nonumber\\
		&\quad\quad+aA{{e}^{-{{h}^{2}}}}\left( \log A-{{h}^{2}} \right),\nonumber
	\end{align}
	or equivalently
	\begin{align}\label{eq:2.2}
		{{h}_{t}}=\Delta h+{{\left| \nabla h \right|}^{2}}\left( \frac{1}{h}-2h \right)-\frac{a}{2}\left( \frac{\log A}{h}-h \right).
	\end{align}
	Using the above equality, we establish the first computational lemma.
	\begin{lemma}\label{eq:l1}
		Let $w={{\left| \nabla h \right|}^{2}}.$  For any $\left( x,t \right)\in {{Q}_{R,T}},$
	\begin{align}\label{eq:2.3}
\Delta w-{{w}_{t}}&\ge -\left[ \max \left\{ 2\left( n-1 \right)K+a,0 \right\}+\frac{1}{{{h}^{2}}}\max \left\{ a\log A,0 \right\} \right]w\nonumber\\
&\quad+2\left( 2h-\frac{1}{h} \right)\left\langle \nabla w,\nabla h \right\rangle +2\left( 2+\frac{1}{{{h}^{2}}} \right){{w}^{2}}.
	\end{align}	
	\end{lemma}
	\begin{proof}
		By the Bochner-Weitzenb\"{o}ck formula, for any function $\psi $, we have
		$$\dfrac{1}{2}\Delta|\nabla \psi|^2= |\nabla^2\psi|^2+\Ric (\nabla \psi,\nabla \psi)+\left\langle\nabla\Delta \psi,\nabla \psi\right\rangle. \nonumber$$
		Therefore, under the assumption $\Ric \geq -\left(n-1\right)K,$ after
		choosing $\psi ={{\left| \nabla h \right|}^{2}}$ we deduce that
		\begin{align}
			{{\Delta }}w-{{w}_{t}}&\ge 2|{{\nabla }^{2}}h{{|}^{2}}-2\left( n-1 \right)K|\nabla h{{|}^{2}}+2\langle \nabla \Delta h,\nabla h\rangle -{{w}_{t}}\nonumber\\
			&\ge -2\left( n-1 \right)Kw+2\langle \nabla \Delta h,\nabla h\rangle -{{w}_{t}}.\nonumber
		\end{align}
		By the equality \eqref{eq:2.2}, we obtain
		\begin{align}\label{eq:2.4}
			{{\Delta }}w-{{w}_{t}}
			&\ge -2\left( n-1 \right)Kw+2\left\langle \nabla \left( {{h}_{t}}+{{\left| \nabla h \right|}^{2}}\left( 2h-\frac{1}{h} \right)+\frac{a}{2}\left( \frac{\log A}{h}-h \right) \right),\nabla h \right\rangle -{{w}_{t}}\nonumber\\
			&\ge -2\left( n-1 \right)Kw+2\left\langle \nabla \left( {{h}_{t}} \right),\nabla h \right\rangle +2\left( 2h-\frac{1}{h} \right)\left\langle \nabla \left( {{\left| \nabla h \right|}^{2}} \right),\nabla h \right\rangle \nonumber\\
			&\quad+2{{\left| \nabla h \right|}^{2}}\left\langle \nabla \left( 2h-\frac{1}{h} \right),\nabla h \right\rangle +a\left\langle \nabla \left( \frac{\log A}{h}-h \right),\nabla h \right\rangle -{{w}_{t}}.
		\end{align}
		Observe that $$2\left\langle \nabla \left( {{h}_{t}} \right),\nabla h \right\rangle ={{\left( |\nabla h{{|}^{2}} \right)}_{t}}={{w}_{t}},$$
		$$\nabla \left( 2h-\frac{1}{h} \right)=2\nabla h+\frac{\nabla h}{{{h}^{2}}}=\left( 2+\frac{1}{{{h}^{2}}} \right)\nabla h,$$
		and
		$$\nabla \left( \frac{\log A}{h}-h \right)=-\log A\frac{\nabla h}{{{h}^{2}}}-\nabla h=-\left( \frac{\log A}{{{h}^{2}}}+1 \right)\nabla h .$$
		Hence, the inequality \eqref{eq:2.4} implies
		\begin{align}
			\Delta w-{{w}_{t}}\ge -\left[ 2\left( n-1 \right)K+a \right]w-\frac{a\log A}{{{h}^{2}}}w+2\left( 2h-\frac{1}{h} \right)\left\langle \nabla w,\nabla h \right\rangle +2\left( 2+\frac{1}{{{h}^{2}}} \right){{w}^{2}}.\nonumber
		\end{align}
		Notice that 
		\begin{align}
		\frac{a\log A}{{{h}^{2}}}&\le \frac{1}{{{h}^{2}}}\max \left\{ a\log A,0 \right\},\nonumber\\
		\left[ 2\left( n-1 \right)K+a \right]&\le \max \left\{ 2\left( n-1 \right)K+a,0 \right\}.\nonumber
		\end{align}
		Therefor, we have 
		\begin{align}
			\Delta w-{{w}_{t}}&\ge -\left[ \max \left\{ 2\left( n-1 \right)K+a,0 \right\}+\frac{1}{{{h}^{2}}}\max \left\{ a\log A,0 \right\} \right]w\nonumber\\
			&\quad+2\left( 2h-\frac{1}{h} \right)\left\langle \nabla w,\nabla h \right\rangle +2\left( 2+\frac{1}{{{h}^{2}}} \right){{w}^{2}}.\nonumber
		\end{align}	
	The proof is complete.
	\end{proof}
	Next, we introduce a smooth cut-off function originated in \cite{LY86} (see also \cite{SZ06, XJ16}).
	\begin{lemma}\label{eq:l2}
	 There exists a smooth cut-off function $\psi =\psi \left( x,t \right)$  supported  in $Q_{R,T},$ satisfying following propositions
		\begin{itemize}
			\item[(i)] $\psi =\psi \left( d\left( x,{{x}_{0}} \right),t \right)\equiv \psi \left( r,t \right);\psi \left( r,t \right)=1$ in ${{Q}_{\frac{R}{2},\frac{T}{2}}},\quad 0\le \psi \le 1.$
			\item[(ii)] $\psi$ is decreasing as a radial function in the spatial variables, and $\frac{\partial \psi }{\partial r}=0$ in ${{Q}_{{R}/{2}\;,T}}.$ 
			\item[(iii)] $\left| \frac{\partial \psi }{\partial t} \right|\frac{1}{{{\psi }^{{1}/{2}\;}}}\le \frac{C}{T}.$
			\item[(iv)] $\left| \frac{\partial \psi }{\partial r} \right|\le \frac{{{C}_{\varepsilon }}{{\psi }^{\varepsilon }}}{{{R}^{2}}}\quad \text{and}\quad \left| \frac{{{\partial }^{2}}\psi }{{{\partial }^{2}}{{r}^{2}}} \right|\le \frac{{{C}_{\varepsilon }}{{\psi }^{\varepsilon }}}{{{R}^{2}}}, \quad\text{when}\quad 0<\varepsilon <1.$
		\end{itemize}
	\end{lemma}
	Now we are ready to prove Theorem 1.3.
	\begin{proof}[Proof of Theorem 1.1]
		  Due to the standard argument of Calabi (see \cite{cal}), we may assume that $\left(\psi w\right)$  obtains its
		  maximal value at  $\left(x_1,t_1\right)$ and we also may assume that that $x_1$ is not on the cut-locus of $M.$
		 At $\left(x_1,t_1\right),$ we have
		$$\nabla \left( \psi w \right)=0,\quad{{\nabla }}\left( \psi w \right)\le 0 \quad\text{and}\quad{{\left( \psi w \right)}_{t}}\ge 0.$$
		Hence, still being at $\left(x_1,t_1\right),$ we obtain
		$$0\ge {{\Delta }}\left( \psi w \right)-{{\left( \psi w \right)}_{t}}=\psi \left( {{\Delta }}w-{{w}_{t}} \right)+w\left( {{\Delta }}\psi -{{\psi }_{t}} \right)+2\left\langle \nabla w,\nabla \psi  \right\rangle.$$
	Using the fact that $$0=\nabla \left( \psi w \right)=w\nabla \psi +\psi \nabla w,$$ we get 
	$$0\ge \psi \left( \Delta w-{{w}_{t}} \right)+w\Delta \psi -w{{\psi }_{t}}-2\frac{{{\left| \nabla \psi  \right|}^{2}}}{\psi }w.$$
This inequality combining with \eqref{eq:2.3} implies
		\begin{align}
	0&\ge -\left[ \max \left\{ 2\left( n-1 \right)K+a,0 \right\}+\frac{1}{{{h}^{2}}}\max \left\{ a\log A,0 \right\} \right]\psi w-2\left( 2h-\frac{1}{h} \right)\left\langle \nabla h,\nabla \psi  \right\rangle w\nonumber\\
		&\quad+2\left( 2+\frac{1}{{{h}^{2}}} \right)\psi {{w}^{2}}+w{{\Delta }}\psi -w{{\psi }_{t}}-2\frac{{{\left| \nabla \psi  \right|}^{2}}}{\psi }w.\nonumber
		\end{align}
		at $\left(x_1,t_1\right).$ In other words, we have just proved that
		\begin{align}
		4\psi {{w}^{2}}&\le \left[ \frac{2{{h}^{2}}}{1+2{{h}^{2}}}\max \left\{ 2\left( n-1 \right)K+a,0 \right\}+\frac{2}{1+2{{h}^{2}}}\max \left\{ a\log A,0 \right\} \right]\psi w\nonumber\\
		&\quad-\frac{4h\left( 1-2{{h}^{2}} \right)}{1+2{{h}^{2}}}\left\langle \nabla h,\nabla \psi  \right\rangle w-\frac{2{{h}^{2}}}{1+2{{h}^{2}}}w{{\Delta }}\psi +\frac{2{{h}^{2}}}{1+2{{h}^{2}}}w{{\psi }_{t}}+\frac{4{{h}^{2}}}{1+2{{h}^{2}}}\frac{{{\left| \nabla \psi  \right|}^{2}}}{\psi }w.\nonumber
		\end{align}
		Since $0<\frac{2{{h}^{2}}}{1+2{{h}^{2}}}\le 1$ and $0<\frac{2}{1+2{{h}^{2}}}\le 2,$ we get
		\begin{align}\label{eq:2.5}
	4\psi {{w}^{2}}\le H\psi w-\frac{4h\left( 1-2{{h}^{2}} \right)}{1+2{{h}^{2}}}\left\langle \nabla h,\nabla \psi  \right\rangle w-\frac{2{{h}^{2}}}{1+2{{h}^{2}}}w\Delta \psi +w\left| {{\psi }_{t}} \right|+\frac{2{{\left| \nabla \psi  \right|}^{2}}}{\psi }w
		\end{align}
		at $\left(x_1,t_1\right).$
   Since $\Ric\ge -\left( n-1 \right)K,$ we can apply the Laplacian comparison theorem to get		
	$$\Delta r\le \frac{n-1}{r}\left( 1+\sqrt{K}r \right).$$
 Using the Laplacian comparison theorem again and Lemma \eqref{eq:l2}  we first have	
\begin{align}\label{eq:2.6}
-\frac{2{{h}^{2}}}{1+2{{h}^{2}}}w\Delta \psi &=-\frac{2{{h}^{2}}}{1+2{{h}^{2}}}w\left[ {{\psi }_{r}}\Delta r+{{\psi }_{rr}}{{\left| \nabla r \right|}^{2}} \right]\nonumber\\&\le \frac{2{{h}^{2}}}{1+2{{h}^{2}}}w\left[ \left| {{\psi }_{r}} \right|\left( \frac{n-1}{r}\left( 1+\sqrt{K}r \right) \right)+\left| {{\psi }_{rr}} \right| \right]\nonumber\\
&\le w\left[ \left| {{\psi }_{r}} \right|\left( \frac{n-1}{r}\left( 1+\sqrt{K}r \right) \right)+\left| {{\psi }_{rr}} \right| \right]\nonumber\\
&\le {{\psi }^{1/2\ }}w\frac{\left| {{\psi }_{rr}} \right|}{{{\psi }^{1/2\ }}}+\left( \frac{n-1}{r}\left( 1+\sqrt{K}r \right) \right){{\psi }^{1/2}}w\frac{\left| {{\psi }_{r}} \right|}{{{\psi }^{1/2}}}\nonumber\\
&\le \frac{3}{5}\psi {{w}^{2}}+c\left[ {{\left( \frac{\left| {{\psi }_{rr}} \right|}{{{\psi }^{1/2\ }}} \right)}^{2}}+{{\left( \frac{n-1}{r}\left( 1+\sqrt{K}r \right)\frac{\left| {{\psi }_{r}} \right|}{{{\psi }^{1/2\ }}} \right)}^{2}} \right]\nonumber\\
&\le \frac{3}{5}\psi {{w}^{2}}+\frac{c}{{{R}^{4}}}+\frac{cK}{{{R}^{2}}}.
\end{align}	
On the other hand, by the Young inequality, we have	
	\begin{align}\label{eq:2.7}
-\frac{4h\left( 1-2{{h}^{2}} \right)}{1+2{{h}^{2}}}\left\langle \nabla h,\nabla \psi  \right\rangle w
&\le 4h\frac{\left| 1-2{{h}^{2}} \right|}{1+2{{h}^{2}}}\left| \nabla \psi  \right|\left| \nabla h \right|w\nonumber\\&\le 4h\left| \nabla \psi  \right|{{w}^{3/2}}=4h\left| \nabla \psi  \right|{{\psi }^{-3/4\ }}{{\left( \psi {{w}^{2}} \right)}^{3/4\ }}\nonumber\\
&\le \frac{3}{5}\psi {{w}^{2}}+ch^4\frac{{{\left| \nabla \psi  \right|}^{4}}}{{{\psi }^{3}}}\nonumber\\&\le \frac{3}{5}\psi {{w}^{2}}+\frac{c{{D}^{2}}}{{{R}^{4}}},
	\end{align}
where $$D=\log A-\log \left( \underset{{{Q}_{R,T}}}{\mathop{\inf }}\,u \right).$$ By using the Cauchy-Schwarz inequality several times, it is not hard for us to see that the following estimates hold true: first for $\psi w$
\begin{align}\label{eq:2.8}
H\psi w\le \frac{3}{5}\psi {{w}^{2}}+c{{H}^{2}},
\end{align}
then for $w\left| {{\psi }_{t}} \right|$ as follows		
	\begin{align}\label{eq:2.9}
w\left| {{\psi }_{t}} \right|&={{\psi }^{{1}/{2}\;}}w\frac{\left| {{\psi }_{t}} \right|}{{{\psi }^{{1}/{2}\;}}}\nonumber
\\&\le \frac{3}{5}{{\left( {{\psi }^{{1}/{2}\;}}w \right)}^{2}}+c{{\left( \frac{\left| {{\psi }_{t}} \right|}{{{\psi }^{{1}/{2}\;}}} \right)}^{2}}\nonumber\\&\le \frac{3}{5}\psi {{w}^{2}}+\frac{c}{T^2}
	\end{align}
and finally for  as the following
\begin{align}\label{eq:2.10}
\frac{2{{\left| \nabla \psi  \right|}^{2}}}{\psi }w&= 2\left( {{\left| \nabla \psi  \right|}^{2}}{{\psi }^{-{3}/{2}\;}} \right)\left( {{\psi }^{{1}/{2}\;}}w \right)\nonumber\\&\le \frac{3}{5}\psi {{w}^{2}}+c\frac{{{\left| \nabla \psi  \right|}^{4}}}{{{\psi }^{3}}}\nonumber\\&\le \frac{3}{5}\psi {{w}^{2}}+\frac{c}{{{R}^{4}}}.
\end{align}		
We now substitute \eqref{eq:2.6}-\eqref{eq:2.10} into the right hand side of \eqref{eq:2.5}, and get that		
\begin{align}\label{eq:2.11}
\psi {{w}^{2}}\le c\left( \frac{{{D}^{2}}+1}{{{R}^{4}}}+\frac{K}{{{R}^{2}}}+\frac{1}{{{T}^{2}}}+{{H}^{2}} \right)
	\end{align}
at $\left(x_1,t_1\right).$ Then, for all $\left( x,t \right)\in {{Q}_{R,T}},$ using \eqref{eq:2.11} we obtain
		\begin{align}
	{{\psi }^{2}}{{w}^{2}}\left( x,t \right)&\le {{\psi }^{2}}{{w}^{2}}\left( {{x}_{1}},{{t}_{1}} \right)\nonumber\\&\le \psi {{w}^{2}}\left( {{x}_{1}},{{t}_{1}} \right)\nonumber\\
	&\le c\left( \frac{{{D}^{2}}+1}{{{R}^{4}}}+\frac{K}{{{R}^{2}}}+\frac{1}{{{T}^{2}}}+{{H}^{2}} \right).
		\end{align}
Notice that $\psi \left( r,t \right)=1$ in ${{Q}_{\frac{R}{2},\frac{T}{2}}}$ and $w={{\left| \nabla h \right|}^{2}},$ we have
\begin{align}\label{eq:1.2}
\frac{\left| \nabla u \right|}{u}\le c\left( n \right)\left( \frac{\sqrt{\log A-\log \left( \underset{{{Q}_{R,T}}}{\mathop{\inf }}\,u \right)}}{R}+\frac{1}{\sqrt{T}}+\sqrt[4]{\frac{K}{{{R}^{2}}}}+\sqrt{H} \right)\sqrt{\log \frac{A}{u}}
\end{align}
in ${{Q}_{\frac{R}{2},\frac{T}{2}}},$ where $c=c(n).$
The proof is complete. 
\end{proof}
We would like to mention that if we let $Q_{R, T}=B(x_0, R)\times[t_0, t_0+T]$ and contruct a similarly test function as in Lemma \ref{eq:l2} then we can obtain a gradient estimate for positive eternal solution to \eqref{eq: 1.2}. Now, we give a verification of Corollary \ref{liouville1}.
\begin{proof}[Proof of Corollary \ref{liouville1}]
By assumption, we have $K=H=0$. Since $u_t=\Delta u$, let $v=u+1$, then $v$ satisfies $v_t=\Delta v$. Moreover, $u, v$ has the same growth at infinity. Hence, without loss of generality, we may assume that $u\geq 1$. Fixing $(x_0, t_0)$ in space-time, Theorem \ref{tr:1} applying to the cube $Q_{R,R}=B(x_0, R)\times [t_0-R, t_0]$ implies that
$$\begin{aligned}
\frac{\left| \nabla u \right|}{u}(x_0, t_0)
&\le c\left( n \right)\left( \frac{\sqrt{\log A}}{R}+\frac{1}{\sqrt{R}}\right)\sqrt{\log \frac{A}{u(x_0, t_0)}}\\
&\leq c\left( n \right)\left( \frac{\sqrt{o(R+|R|)}}{R}+\frac{1}{\sqrt{R}}\right)\sqrt{ o(R+|R|)-\log(u(x_0, t_0))}.
\end{aligned}$$ 
Let $R$ goes to infinity, it turns out that $|\nabla u(x_0, t_0)|=0$. Since $(x_0, t_0)$ is arbitrary, we conclude that $u$ is constant.
\end{proof}
Theorem \ref{tr:1} in fact implies  Corollary \ref{liouville2}.
\begin{proof}[Proof of Corollary \ref{liouville2}]
Note that if $a\le 0$ and $1\le u \le A$ then the inequality \eqref{eq: 1.2}  can be rewritten as following 
\begin{align}\label{eq: 2.14}
\frac{\left| \nabla u \right|}{u}\le c\left( n \right)\left( \frac{\sqrt{\log A}}{R}+\frac{1}{\sqrt{T}}+\sqrt[4]{\frac{K}{{{R}^{2}}}}+\sqrt{\max \left\{ 2\left( n-1 \right)K+a,0 \right\}} \right)\sqrt{\log \frac{A}{u}}.
\end{align}
Since $a\le -2\left(n-1\right)K,$ we get 
$$\max \left\{ 2\left( n-1 \right)K+a,0 \right\}=0.$$
Letting $R\to +\infty, T \to +\infty$ in \eqref{eq: 2.14}, we obtain $u$ is  a constant. Using $\Delta u+au\log u=0,$ we get $u\equiv 1.$ \end{proof}

\begin{remark}\label{remark}
	Let $\left( M,g,{{e}^{-f}}dv \right)$ be an $n$-dimensional complete smooth metric measure space with 
$$\Ric_{f}\ge -\left( n-1 \right)K$$
	for some constant $K\ge 0$ in $B\left(x_0, R\right)$, for some fixed point $x_0$ in $M,$ and some fixed
	radius $R\ge 2.$ Assume that $0<u\left(x,t\right)< A$ for some constant $A,$ is a smooth solution to the general $f$-heat equation. 
\begin{equation}\label{weightedequation}
u_t=\Delta_fu+au\log u,
\end{equation}	
 in the cylinder ${{Q}_{R,T}}:\equiv B\left( {{x}_{0}},R \right)\times \left[ {{t}_{0}}-T,{{t}_{0}} \right]\subset M\times \left( -\infty ,\infty  \right),$ where $a\in\mathbb{R}$ is fixed, $ t_0\in \mathbb{R}$ and $T>0$. Following the proof of Theorem \ref{tr:1} and using Laplacian comparison theorem as in \cite{Bri} instead of classical Laplacian comparison theorem, we can show that there exists a constant $c\left(n\right)$ such that
	\begin{align}
\frac{\left| \nabla u \right|}{u}\le c\left( n \right)\left( \frac{\sqrt{\log A-\log \left( \underset{{{Q}_{R,T}}}{\mathop{\inf }}\,u \right)}+\sqrt{\left| \alpha  \right|R}}{R}+\frac{1}{\sqrt{T}}+\sqrt{K}+\sqrt{H} \right)\sqrt{\log \frac{A}{u}}.\label{weiver}
	\end{align}
in ${{Q}_{\frac{R}{2},\frac{T}{2}}},$ where 
$$H=\max \left\{ a,0 \right\}+\max \left\{ 2a\log A,0 \right\}.$$
Here, $\alpha :={{\max }_{\left\{ x\left| d\left( x,{{x}_{0}} \right)=1 \right. \right\}}}{{\Delta }_{f}}r\left( x \right),$ where $r(x)$ is the distance from $x$ to $x_0.$ Since the proof of this gradient estimate is similar to the proof of Theorem \ref{tr:1}, we omit the detail. 
\end{remark}
From this remark, we observe that if $\Ric_f\geq0$, let $u$ be a positive soltion to \eqref{weightedequation}, where $a\leq0$ and $u=e^{o(d(x)+|t|)}$ then $u$ must be constant, namely $u\equiv1$. The Example 1.2 in \cite{Wu15} implies that the estimate \eqref{weiver} is sharp. Indeed, let us recall the example.
\begin{example}For any $a, b>0$ let 
$$u=e^{ax+(a^2+ab)t}, \quad f=-bx$$
then $u$ satisfies the heat equation $u_t=\Delta_fu$.
\end{example}  
This example confirms that our gradient estimate is sharp in both spatial and time directions. In \cite{Wu15}, to obtain a Liouville property, the author had to require that the potential function $f$ is bounded and $u=e^{o(d(x)+\sqrt[]{|t|})}$.
	\vskip 0.3cm
	\noindent
\section*{Acknowledgment } This work was initiated when second author visited ICTP to participate in the international school on Extrinsic Geometric Flows. He would like to thank the staff there for hospitality and financial support. This work was partialy supported by NAFOSTED under grant number 101.02-2017.313.

	\bigskip
\address{{\it Ha Tuan Dung}\\
Faculty of Mathematics  \\
Hanoi Pedagogical University No. 2  \\
Xuan Hoa, Vinh Phuc, Vietnam 
}
{hatuandung.hpu2@gmail.com}
\address{ {\it Nguyen Thac Dung}\\
Faculty of Mathematics - Mechanics - Informatics \\
Hanoi University of Science (VNU) \\
Ha N\^{o}i, Vi\^{e}t Nam 
}
{dungmath@gmail.com}

\end{document}